\documentclass[10pt, a4paper]{amsart}
\usepackage{amsmath,amssymb}
\usepackage{amsthm}
\usepackage[pdftex]{graphicx}
\DeclareGraphicsExtensions{.png,.pdf,.jpg}
\usepackage{color}

\let\abs=\envert

\newcommand{\PUNTOS}{\leaders\hbox to 16.6pt{.\hss}\hfill}
\makeatletter
\newcommand{\puntitos}{\cleaders\hbox{$\m@th \mkern4mu.\mkern4mu$}\hfill}
\newcommand{\puntillos}{\leaders\hbox{$\m@th \mkern3mu.\mkern3mu$}\hfill}
\makeatother

\newtheorem{teor}{Theorem}

\theoremstyle{definition}
\newtheorem{defi}[teor]{Definition}
\newtheorem{corol}[teor]{Corollary}
\newtheorem{defis}[teor]{Definitions}
\newtheorem{lema}[teor]{Lemma}

\newtheorem{propi}[teor]{Property}

\DeclareMathOperator{\orden}{order}

\begin{document}
\title[CZ decomposition and differentiation of integrals on $\mathbb{T}^\omega$]{A decomposition of Calder\'on--Zygmund type and some observations on differentiation of integrals on the infinite-dimensional torus}
\author[E. Fern\'andez]{Emilio Fern\'{a}ndez} 
\address{Departamento de Matem\'{a}ticas y Computaci\'{o}n, 
Universidad de La Rioja, c/ Madre de Dios, 53, 26006 Logro\~no, Spain.}
\email{emfernan@unirioja.es}

\author[L. Roncal]{Luz Roncal}
\address{BCAM -- Basque Center for Applied Mathematics, 48009 Bilbao, Spain and Ikerbasque,
Basque Foundation for Science, 48011 Bilbao, Spain}
\email{lroncal@bcamath.org}

\thanks{The second author is supported by the Basque Government through the BERC 2018-2021 program, by Spanish Ministry of Economy and Competitiveness MINECO: BCAM Severo Ochoa excellence accreditation SEV-2017-2018  and through project MTM2017-82160-C2-1-P funded by (AEI/FEDER, UE) and acronym ``HAQMEC'', and by a 2017 Leonardo grant for Researchers and Cultural Creators, BBVA Foundation. The Foundation accepts no responsibility for the opinions, statements and contents included in the project and/or the results thereof, which are entirely the responsibility of the authors.}



\begin{abstract}

In this note we will show a Calder\'on--Zygmund decomposition associated with a function $f\in L^1(\mathbb{T}^{\omega})$. The idea relies on an adaptation of a more general result by J. L. Rubio de Francia in the setting of locally compact groups. Some related results about differentiation of integrals on the infinite-dimensional torus are also discussed. 
\end{abstract}

\maketitle
\thispagestyle{empty}

\section{Introduction}

In  \cite{rdf78}, Jos\'{e} L. Rubio de Francia (JLR) showed a result on differentiation of integrals in the context of a locally compact group $G$, 
that contained a decomposition of Calder\'on--Zygmund type \cite[Ch. I, Lemma 1]{CZdecomp} under certain conditions. In view of the publication date, his study could be a contemporary of the one by Edwards and Gaudry in \cite[Ch. 2]{edwardsgau}. In this note we revisite and adapt those results by JLR to the case of the (compact, abelian) group $\mathbb{T}^\omega$ (the infinite torus) defined below. Indeed, our goals are the following:
\begin{enumerate}
\item Present a decomposition of Calder\'{o}n--Zygmund (CZ) type  in $\mathbb{T}^\omega$, devised by JLR.
\item Observe some issues on differentiation of integrals in $\mathbb{T}^\omega$.
\end{enumerate}

Let $\mathbb{T}=\{z\in\mathbb{C}\,\vert\, \abs{z}=1\}$ be the one-dimensional torus, identified naturally with the interval $[0,1)$ of the real line through
the group isomorphism  $e^{2\pi it}\longleftrightarrow t$. We can also identify $\mathbb{T}$ with the additive quotient group $\mathbb{R}/\mathbb{Z}$.
We denote by $\mathbb{T}^\omega$ 
the compact group formed by the product of countably infinite many copies of $\mathbb{T}$ (\emph{complete direct sum}, \cite[B7.]{rudinfg}). We will call it briefly the \emph{infinite torus}. The operation in the group $\mathbb{T}^\omega$ is the sum (mod 1) of real sequences, with identity element $\bar{0}:=(0,0,\ldots)$. 

For a fixed $n\in\mathbb{N}$ we can write the infinite torus, of points $x=(x_1,x_2,\ldots)$, as the cartesian product
$\mathbb{T}^n\times\mathbb{T}^{n,\omega}$ of an $n$-dimensional torus $\mathbb{T}^n$ of points $x_{(n}=(x_1,\ldots,x_n)$ 
times a copy, but denoted $\mathbb{T}^{n,\omega}$, of the infinite torus itself, of points $x^{(n}=(x_{n+1},x_{n+2},\ldots)$.

We denote by $m$, or $dx$, the Haar measure (translation invariant) on $\mathbb{T}^\omega$, normalized such that $m(\mathbb{T}^\omega)=1$. This measure coincides \cite[\S 22]{hewitstrom} with the measure product of countably infinite many copies of the Lebesgue measure $|\cdot|$  on $\mathbb{T}$, so the basic $m$-measurable sets are the so-called \emph{intervals}, i.e. subsets of $\mathbb{T}^\omega$ of the form $I=\prod_{j\in\mathbb{N}} I_j$, where $I_j$ is an interval of $\mathbb{T}$ for each  $j$, and $\exists\,N\in\mathbb{N}$ such that  $I_{j}=\mathbb{T}$ for all $j>N$. The measure of the interval $I$ is then $m(I)=\int_{\mathbb{T}^\omega} \chi_I(x)\,dx=\prod_{j=1}^N |I_j|$.

The space $\mathbb{T}^\omega$ is metrizable. For instance, the function 
\begin{equation}\label{metrics}
d(x,y)=\sum_{n=1}^\infty\frac{|x_n-y_n|}{2^n}\quad (x,y\in\mathbb{T}^\omega)
\end{equation}
defines a metric in $\mathbb{T}^\omega$ \cite[p. 157]{saks}.
We write $\delta(S):=\sup_{x,y\in S}d(x,y)$ for the diameter of the set  $S\subset\mathbb{T}^\omega$. 

The $\sigma$-algebra $\mathcal{B}$ of Borel sets in $\mathbb{T}^\omega$, which is the smallest $\sigma$-algebra containing the open intervals, coincides with the least $\sigma$-algebra containing the open balls with respect to the metric \eqref{metrics} \cite[II \S 2.4]{shiry}.

The study of Harmonic Analysis on the infinite torus finds a motivation, on one hand, because it constitutes a logical extension of the $n$-dimensional setting in which estimates have to be obtained independent of the dimension $n$. On the other hand, $\{e^{2\pi i x_k}\colon k=1,2,\ldots\}$ is a system of independent random variables uniformly distributed in the complex unit circumference (i.e., of a complexified version of Rademacher's functions), whose natural completion in $L^2$ is the trigonometric system on $\mathbb{T}^\omega$. Then, the Fourier series of infinitely many variables turn out to be the complex analogue of the Walsh series. Fourier series in $\mathbb{T}^\omega$ also have connection with the Dirichlet series \cite{bohr} and with Prediction Theory \cite{helsonlow}. All these issues were already pointed out by JLR in \cite{rdf80} (see also references therein), where he studied pointwise and norm convergence of Fourier series of infinite variables, although the proofs are just sketched.

There is considerable interest on the infinite torus from the point of view of Potential Theory, see \cite{bendikov,b3, b1,b2,berg}. Apart from this, problems of approximation theory on $\mathbb{T}^\omega$ have been analyzed for instance in \cite{Platonov}. 

The JLR decomposition of Calder\'on--Zygmund type in $\mathbb{T}^\omega$ will be shown in Section \ref{sec:condi} (see Subsection \ref{subsec:CZdes}), and the issues related to differentiation of integrals in  $\mathbb{T}^\omega$ are contained in Section \ref{sec:dif}. To be precise, we will look at three differentiation bases. First, the \textit{Rubio de Francia restricted  basis} $\mathcal{R}_0$, which is the family associated to the Calder\'on--Zygmund decomposition in Section  \ref{sec:condi} and that differentiates $L^1(\mathbb{T}^\omega)$, see Corollary \ref{cor:RDFr}. Second, the \textit{Rubio de Francia basis} $\mathcal{R}$, which arises naturally in the light of the general results by JLR in \cite[Thm. 8]{rdf78}. For such a basis several questions remain open concerning differentiation and the associated maximal function, see Subsection \ref{subsec:DT}. Finally, we present a negative result of differentiation on $\mathbb{T}^\omega$ relative to
the so called \textit{extended Rubio de Francia basis} $\mathcal{R^\ast}$, see Subsection \ref{subsec:neg}. 

\section*{Acknowledgments}

The original idea of the Calder\'on--Zygmund (CZ) decomposition presented in Section \ref{sec:condi} was sketched in a personal communication of JLR to the first author in 1977, in Madrid. 

The authors would like to thank the referees for their very careful reading and
useful comments which indeed improved the presentation of the paper.


\section{A CZ decomposition  in $\mathbb{T}^\omega$}
\label{sec:condi}
For completeness we will 
recall some concepts from Probability Theory used later in this section (see for instance \cite[p. 89--94]{steintopics}, \cite[Ch. 5]{edwardsgau}, \cite[2.7]{shiry}).

\begin{defi}
Let $(X,\mathcal{A},\mu)$ be a finite measure space, and $\mathcal{B}$ be a $\sigma$-algebra contained in $\mathcal{A}$. 
The \textit{conditional expectation of $f$ given $\mathcal{B}$}
is the (unique $\mu$-a.e.) 
$\mathcal{B}$-measurable function $E^{\mathcal{B}}f$ (the notation is that of \cite{neveu}), such that
\begin{equation}\label{especond}
\int_B f\,d\mu = \int_B (E^{\mathcal{B}}f)\,d\mu  \quad\forall\, B\in\mathcal{B}.
\end{equation}
E.g., suppose that $\{B_n\}_{n=1}^\infty$  is a countable  division of $X$ in $\mathcal{A}$-measurable sets of positive measure, and consider the least $\sigma$-algebra  $\mathcal{B}$ which contains those sets (we write $\mathcal{B}:=\sigma(\{B_n\})$. 
Then,
\begin{equation}
\label{comparfm1}
E^{\mathcal{B}}f(x)=\sum_n f_{B_n}\chi_{B_n}(x)
\end{equation}
(where $f_{B}:=\frac1{\mu(B)}\int_{B} f\,d\mu$ and $\chi_S$ denotes the characteristic function of the set $S$), since
the  function $s(x)$ at the right hand side of \eqref{comparfm1} is $\mathcal{B}$-measurable and 
$\int_{B_n}s\,d\mu=f_{B_n} \mu(B_n)=\int_{B_n}f\,d\mu$ holds.
\end{defi}
\begin{propi}\label{propi2}
If $\mathcal{B}\subset \mathcal{C}$ are sub-$\sigma$-algebras of $\mathcal{A}$, then  $E^{\mathcal{B}}(E^{\mathcal{C}}f)=E^{\mathcal{B}}f$ a.e. 
\end{propi}

\begin{defi} Let $(X,\mathcal{A},m)$ be a finite measure space and let
$$
\mathcal{B}_1\subset \mathcal{B}_2\subset \cdots\subset \mathcal{B}_n\subset \mathcal{B}_{n+1}\subset \cdots
$$ 
be an increasing sequence of sub-$\sigma$-algebras of $\mathcal{A}$.  

A  sequence of functions $\{f_n\}_{n\in\mathbb{N}}\subset L^1(m)$ such that, for each $n\ge1$ the function  $f_n$ is $\mathcal{B}_n$-measurable and $E^{\mathcal{B}_n}f_{n+1}=f_n$ (a.e.), is called a \emph{martingale}.
\end{defi}

For instance, for every $f\in L^p(\mu)$ ($1\le p\le \infty$) the sequence $f_n:=E^{\mathcal{B}_n}f$ ($n\in\mathbb{N}$) is a martingale,
since $E^{\mathcal{B}_n}f_{n+1}=E^{\mathcal{B}_n}(E^{\mathcal{B}_{n+1}}f)=E^{\mathcal{B}_n}f$ a.e., according to Property~\ref{propi2}. Moreover the following holds: 
\begin{teor} \label{maxmartinteo} \emph{(i)} The maximal operator, associated to $\{\mathcal{B}_n\}$,  defined on $L^1(\mu)$ by $E^\ast f(x):=\sup_{n}|f_n(x)|$, where $f_n=E^{\mathcal{B}_n}f$ ($n\in\mathbb{N}$), is weak $(1,1)$ \emph{(Doob's inequality \cite[VII, Thm. 3.2]{doob})}, and strong $(p,p)$, $1<p\le\infty$. 

\emph{(ii)} Furthermore,
$(f_n)$ converges almost everywhere. Actually, 
$$
\lim_{n\to\infty}f_n(x)=(E^\mathcal{B} f)(x) \ \text{$\mu$-a.e.,}
$$
where  $\mathcal{B}=\sigma\bigl(\bigcup_{n=1}^\infty \mathcal{B}_n\bigr)$.
\end{teor}

\subsection{JLR on decomposition of CZ type in locally compact groups}

Let $G$ be a locally compact group 
with identity $e$ and Haar measure (left invariant) $m$, and $H$ be a discrete subgroup of $G$. We will first give  a definition and a lemma.
\begin{defi} (\cite[Section 1]{rdf78}.) \label{funddom} An open subset $V$ of $G$ is called a \emph{fundamental domain} (FD) for the quotient group $G/H$ if these two conditions hold:
\begin{enumerate}
\item $VV^{-1}\cap H=\{e\}$ (or what is the same, the restriction $\pi|_V$ of the canonical projection $\pi\colon G\to G/H$ is $1-1$).
\item 
The complement of $VH$ in $G$ is a locally null set.\footnote{Cf. \cite[(20.11) Definition]{hewitstrom}. }
\end{enumerate}
\end{defi}
For example,  the open interval $(0,1)$ is a FD for $\mathbb{R}/\mathbb{Z}$.  For each $n\in\mathbb{N}$, the interval  $\bigl(0,\frac1n\bigr)$ is a FD for $\mathbb{T}/R_n$, where $R_n:=\bigl\{0,\frac1n,\ldots,\frac{n-1}n\bigr\}$  is the subgroup of the $n$th roots of unity in $\mathbb{T}$. 

\begin{lema} (\cite[Lemma 2]{rdf78}.)
Assume that $G$ contains a sequence of discrete subgroups
$$
H_1\subset H_2\subset \cdots\subset H_n\subset \cdots\subset G
$$
such that each $G/H_n$ is compact, and write $k_n:=\orden(H_{n+1}/ H_n)$.
Then, there is a sequence of open sets
$$
V_1\supset V_2\supset \cdots\supset V_n\supset \cdots\text{,}
$$
such that $V_n$ is a FD for $G/H_n$, and each $V_n$ is,
except for a null set, the disjoint union of $k_n$ translates of $V_{n+1}$ by elements of $H_{n+1}$. 
\end{lema}

The main result of JLR on decomposition of CZ type in this context, is the following:

\begin{teor} \emph{(\cite[Thm. 8, into the first part of the proof]{rdf78}.)} \label{teor8jlr78}
Assume additionally that $\cup_n H_n$ is dense in $G$ and that
\begin{equation}
\label{jlr78cond1}
\sup_n \,k_n=k<\infty.
\end{equation}
Then, for each $f\in L^1(G)$ and $a>\|f\|_1$ there is a disjoint sequence of open sets $S_j$ belonging to the family $\{tV_n\colon t\in H_n\text{,}\  n\in\mathbb{N}\}$,  such that $|f(x)|\le a$ a.e. outside $A=\cup_j S_j$,  $m(A)\le C\|f\|_1/a$ for a constant $C$ independent from $f$ and $a$, and $a\le |f|_{S_j}\le ka$ $(j=1,2,\ldots)$.
\end{teor} 


\subsection{A decomposition of CZ type in $\mathbb{T}^\omega$}
\label{subsec:CZdes}

\quad

\smallskip
Our subsequent Theorem \ref{descompCZ} will show the original proof of JLR Theorem \ref{teor8jlr78} in case of the group (compact, abelian) $G=\mathbb{T}^\omega$. Before its statement  we must establish a suitable sequence of subgroups (that JLR did actually teach us in the aforementioned personal communication), which we present below. The decomposition of CZ type in $\mathbb{T}^\omega$ will turn out to be associated with a certain family $\mathcal{R}_0$ (see \eqref{redtraslad}) of ``dyadic  intervals".

\begin{defis} (\cite[VII.43]{munroe}.) A \emph{net}\footnote{Do not confuse it with net in the sense \cite[Ch. 2]{kelley} of directed set or generalized sequence. The concept of  sequence of nets, originally in the euclidean space, 
is due to de la Vall\'{e}e Poussin \cite[10.67]{poussin15}. Saks \cite[p. 153]{saks} generalizes it to measure metric spaces.  See also \cite[\S 6]{jessen}.} 
in $\mathbb{T}^\omega$ is a countable class of disjoint measurable sets whose union is $\mathbb{T}^\omega$ except for a set of null measure.  Let $\{\mathcal{M}_n\}_{n\in\mathbb{N}}$ be a sequence of nets. 
The sequence is called \emph{monotonic} if for each positive integer $n$, every set of $\mathcal{M}_{n+1}$ is a subset of some set of $\mathcal{M}_n$. In this case, for almost all $x$ there exists, for each $n\in\mathbb{N}$, an unique set
$I_x^{(n)}\in\mathcal{M}_n$ such that $x\in I_x^{(n)}$.  
\end{defis}

Remember that $R_k:=\{0,\frac1k,\ldots,\frac{k-1}k\}$, $k\in \mathbb{N}$.  In the following table we see the first terms of the increasing sequence of subgroups $H_m\subset\mathbb{T}^\omega$ proposed by JLR, as well as some of the first terms of the associated decreasing sequence of FD:
$$
\begin{array}{r|l|l}
m&H_m&V_m\\\hline
1^2&H_1=R_2\times\{\overline{0}^{(1}\}&V_1=(0,\frac 12)\times \mathbb{T}^{1,\omega}\\
1^2+1&H_2=R_2\times R_2\times\{\overline{0}^{(2}\}&V_2=(0,\frac 12)^2\times \mathbb{T}^{2,\omega}\\
&H_3=R_4\times R_2\times\{\overline{0}^{(2}\}&\\
2^2&H_4=R_4\times R_4\times\{\overline{0}^{(2}\}&V_4=(0,\frac 14)^2\times \mathbb{T}^{2,\omega}\\
&H_5=R_4\times R_4\times R_2\times\{\overline{0}^{(3}\}&\\
2^2+2&H_6=R_4\times R_4\times R_4\times\{\overline{0}^{(3}\}&V_6=(0,\frac 14)^3\times \mathbb{T}^{3,\omega}\\
&H_7=R_8\times R_4\times R_4\times\{\overline{0}^{(3}\}&\\
&H_8=R_8\times R_8\times R_4\times\{\overline{0}^{(3}\}&V_{8}=(0,\frac 1{8})^2\times(0,\frac14)\times \mathbb{T}^{3,\omega}\\
3^2&H_9=R_8\times R_8\times R_8\times\{\overline{0}^{(3}\}&V_9=(0,\frac 18)^3\times \mathbb{T}^{3,\omega}\\
&H_{10}=R_8\times R_8\times R_8\times R_2\times\{\overline{0}^{(4}\}&\\
\puntillos&\puntillos&\puntillos
\end{array}
$$

Actually, after $H_1=R_2\times\{\bar{0}^{(1}\}$ we define, for each $n\ge 1$ 
\begin{gather*}
H_{n^2+j}=\widetilde{H}_{n^2+j}\times\{\bar{0}^{(n+1}\},\quad (1\le j\le 2n+1),\\
\intertext{and}
\widetilde{H}_{n^2+j}
:=\begin{cases}
R_{2^n}\times\overset{(n)}{\cdots}\times R_{2^n}\times R_{2^j} \qquad \,\qquad\qquad \qquad \text{if $j\in\{1,\ldots,n\}$},\\[4pt]
R_{2^{n+1}}\times\overset{(j-n)}{\cdots}\times R_{2^{n+1}}\times  R_{2^n}\times\overset{(2n+1-j)}{\cdots}\times R_{2^n}\\[4pt]
\qquad \qquad \qquad \qquad \qquad \qquad \qquad \qquad \qquad  \text{if $j\in\{n+1,\ldots,2n+1\}$}.
\end{cases}
\end{gather*}

For each positive integer $m$, $H_m$ is a discrete (finite, of order $2^m$) subgroup  of $\mathbb{T}^\omega$, $H_m\subset H_{m+1}$, and 
$\orden(H_{m+1}/H_m)=2$ for all $m\ge 1$. Moreover, each $\mathbb{T}^\omega/H_m$ is compact, because $\mathbb{T}^\omega$ is, and  the union $\bigcup_m H_m$ is a dense subset of $\mathbb{T}^\omega$ as is easily checked.

\begin{figure}[ht]
\begin{center}
\quad\hspace*{-30pt}\includegraphics[scale=.45]{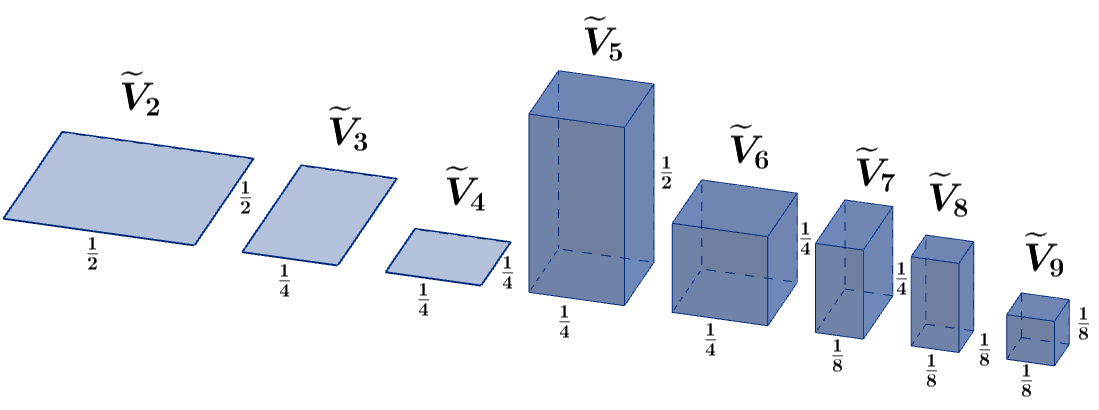}
\caption{First members of the sequence $\{\widetilde{V}_m\}$. E.g., $\widetilde{V}_7=(0,\frac 18)\times(0,\frac 14)^2$.
Two translations of $V_{m+1}$ by elements of $H_{m+1}$ cover (a.e.) $V_m$.}\label{CZ_fig1}
\end{center}
\end{figure}
The  associated decreasing sequence of open sets $\{V_m\}$, where for each $m\ge 1$, $V_m$ is a FD for $\mathbb{T}^\omega/H_m$, is defined by $V_1=(0,\tfrac12)\times \mathbb{T}^{1,\omega}$
and, for each $n\ge 1$,
\begin{equation}
\label{uvesubn}
V_{n^2+j}=\widetilde{V}_{n^2+j}\times\mathbb{T}^{n+1,\omega}\quad (1\le j\le 2n+1),
\end{equation}
with
$$
\widetilde{V}_{n^2+j}
:=\begin{cases}
(0,\tfrac1{2^{n}})^n\times (0,\tfrac1{2^{j}})&\text{if $j\in\{1,\ldots,n\}$},\\[4pt]
(0,\tfrac1{2^{n+1}})^{j-n}\times (0,\tfrac1{2^{n}})^{2n+1-j} &\text{if $j\in\{n+1,\ldots,2n+1\}$}
\end{cases}
$$
(see Figure \ref{CZ_fig1}).

We can consider the (finite) net $\mathcal{N}_m:=\{t+V_m\colon t\in H_m\}$ for each $m\in\mathbb{N}$.   
The sequence $\{\mathcal{N}_m\}_{m\in\mathbb{N}}$ is monotonic. Our final family $\mathcal{R}_0$ of dyadic intervals in $\mathbb{T}^\omega$ is the union of this monotonic sequence of nets, 
\begin{equation}\label{redtraslad}
\mathcal{R}_0:=\bigcup_m \mathcal{N}_m=\left\{t+V_m\colon  m\in\mathbb{N}, t\in H_m \right\}.
\end{equation}

The following announced result holds. The collection $\{I_j\}$ will be a \emph{Calder\'on-Zygmund decomposition} of intervals of  $\mathcal{R}_0$ 
associated with the function $f$ at level~$a$.

\begin{teor}[Jos\'{e} L. Rubio de Francia] \label{descompCZ} 
Let $f\in L^1(\mathbb{T}^\omega)$ and $a> \|f\|_1$. There exists a closed set $F_a$ and an open set $\Omega_a=\mathbb{T}^\omega\setminus F_a$ such that
\begin{enumerate}
\item[(i)] 
$|f(x)|\le a$ for almost all $x\in F_a$.
\item[(ii)] 
The set $\Omega_a$ is the union of a sequence $\{I_j\}_{j=1}^\infty$ of pairwise disjoint intervals 
of the family $\mathcal{R}_0$ 
such that $a\le |f|_{I_j}\le 2a$ for all $I_j$.
\item[(iii)] It is verified $m(\Omega_a)\le \frac{\|f\|_1}a$.
\end{enumerate}
\end{teor}

\begin{proof}
(See also \cite[Thm. 2.10 and 2.11]{duoandiko}.)

For each $m\in\mathbb{N}$, let $\mathcal{B}_m:=\sigma(\mathcal{N}_m)$. 
Consider as well the trivial $\sigma$-algebra $\mathcal{B}_0=\{\emptyset,\mathbb{T}^\omega\}$ (which is generated by the open set $V_0:=\mathbb{T}^\omega=\bar{0}+\mathbb{T}^\omega$ and thus has the above general form if we also consider the trivial subgroup $H_0=\{\bar{0}\}$). We have
$$
\mathcal{B}_0\subset\mathcal{B}_1 \subset \mathcal{B}_2\subset \cdots \mathcal{B}_m\subset \mathcal{B}_{m+1}\subset \cdots.
$$
Define, for $m=0,1,2,\ldots$
\begin{equation}\label{comparfm2}
f_m(x):=\sum_{t\in H_m}|f|_{t+V_m}\cdot\chi_{t+V_m}(x).  
\end{equation}
The function $f_m$ is $\mathcal{B}_m$-measurable for each $m$, since it is constant on each interval  $t+V_m$ component of the net $\mathcal{N}_m$. Moreover, $f_m\ge 0$ and
$$
\int_{\mathbb{T}^\omega}f_m=\sum_{t\in H_m}|f|_{t+V_m}\cdot m(V_m)=\sum_{t\in H_m}\int_{t+V_m}|f|=\int_{\mathbb{T}^\omega}|f|=\|f\|_1<\infty
$$
by hypothesis (we have used that $\mathbb{T}^\omega\setminus(H_m+V_m)$ is a null set, since  $V_m$ is a FD for $\mathbb{T}^\omega/H_m$ for each $m$). If we compare \eqref{comparfm2} with \eqref{comparfm1} we will see that $f_m$ is the conditional expectation of $|f|$ related to the $\sigma$-algebra $\mathcal{B}_m$. 

By applying Theorem \ref{maxmartinteo}(ii) we deduce that $f_m(x)\to (E^{\mathcal{B}}|f|)(x)$ a.e. as $m\to\infty$, $\mathcal{B}$ being the least $\sigma$-algebra containing the union $\bigcup_{m=0}^\infty \mathcal{B}_m$. In our case, the set  
$\bigcup_m H_m$ is dense in  $\mathbb{T}^\omega$, and thus the class $\bigcup_{m=0}^\infty \mathcal{B}_m$ is the class of the open sets in $\mathbb{T}^\omega$. Then $\mathcal{B}$ is the $\sigma$-algebra of Borel sets in $\mathbb{T}^\omega$. Therefore the operator $E^{\mathcal{B}}$ of conditional expectation with respect to the $\sigma$-algebra $\mathcal{B}$ is the identity, and so, for all $g\in L^1(\mathbb{T}^\omega)$ it is verified that $E^{\mathcal{B}}g=g$ (a.e.). We conclude that
\begin{equation}\label{convaemartin}
\lim_{m\to\infty}f_m(x)=\abs{f(x)} \quad\text{a.e. in $\mathbb{T}^\omega$}.
\end{equation}
Let
$$
f^*(x)=\sup_{m\in\mathbb{N}}f_m(x).
$$  
If $x\in F_a:=\{x\colon f^*(x)\le a\}$, we will have $f_m(x)\le a $ for all $m$, and an application of \eqref{convaemartin} yields $\abs{f(x)}\le a$, so part (i) is proven.

The set $\Omega_a=\mathbb{T}^\omega\setminus F_a$ from part (ii) is defined as $\Omega_a =\{x\colon f^*(x)> a\}$, and 
the weak type $(1,1)$ inequality (Theorem \ref{maxmartinteo}(i)) 
for the maximal operator\footnote{According to Definitions \ref{defbasedepossel} below, the operator $E^\ast$ would be denoted as $M^{\mathcal{R}_0}$. } 
$E^\ast\colon f\mapsto f^*$ gives
\begin{equation}\label{jlweak11}
m(\Omega_a)\le \frac Aa\|f\|_1,
\end{equation}
where $A$ is a constant independent of $f$ and $a$.

Finally, we have supposed that  $\|f\|_1 < a$, thus $f_0(x)\le a$ holds for all $x$, and we can recognise the set $\Omega_a$ as the disjoint union of the sets
$$
\Omega_a^{(n)}=\{x\colon f_i(x)\le a<f_n(x),\ 0\le i\le n-1\}, \quad n=1,2,\ldots.
$$
For each $n\ge 1$, the set $\Omega_a^{(n)}$ is obviously $\mathcal{B}_n$-measurable, therefore it is the disjoint union of intervals of the form $t+V_n$ with $t\in H_n$. If $I_j^{(n)}$ is one of these intervals we have, on one hand,
\begin{align}
\frac1{m\bigl(I_j^{(n)}\bigr)}\int_{I_j^{(n)}}|f(x)|\,dx&=
\frac1{m\bigl(I_j^{(n)}\bigr)}\int_{I_j^{(n)}}\bigl(E^{\mathcal{B}_n}|f|\bigr)(x)\,dx\label{czdetail}\\
&=\frac1{m\bigl(I_j^{(n)}\bigr)}\int_{I_j^{(n)}}f_n(x)\,dx\ge a\notag
\end{align}
because $f_n(x)>a$  $\forall x\in I_j^{(n)}\subset \Omega_a^{(k)}$. On the other hand, $I_j^{(n)}$ is contained in an interval of the form $s+ V_{n-1}$ ($s\in H_{n-1}$) which is not contained in $\Omega_a^{(n-1)}$ (we make the agreement that $\Omega_a^{(0)}=\emptyset$), so it is $f_{n-1}(x)\le a$ for all $ x\in s+ V_{n-1}$. By using also that $m\bigl(I_j^{(n)}\bigr)=m(V_n)=\frac12m(V_{n-1})$, we have
\begin{align*}
\frac1{m\bigl(I_j^{(n)}\bigr)}\int_{I_j^{(n)}}|f(x)|\,dx&\le \frac2{m(V_{n-1})}\int_{s+V_{n-1}}\abs{f(x)}\,dx\\ 
&=\frac2{m(V_{n-1})}\int_{s+V_{n-1}}\bigl(E^{\mathcal{B}_{n-1}}|f|\bigr) (x)\,dx\\
&= \frac2{m(V_{n-1})}\int_{s+V_{n-1}}f_{n-1}(x)\,dx\le 2a,
\end{align*} 
which finishes the proof of (ii). Moreover, from \eqref{czdetail} it follows that $A=1$ in \eqref{jlweak11}.
\end{proof}

\medskip
\noindent\textbf{Remarks.} 

\smallskip
(1) It is well known that the standard use of the CZ decomposition of the open set $\Omega_a=\cup_j I_j$ involves \cite[Ch. I, proof of Lemma 2]{CZdecomp}, also in $\mathbb{T}^\omega$, a decomposition of the function $f$, at each level $a$, in the sum of
\begin{equation*}
g(x):=f(x)\chi_{F_a}(x)+\sum_j f_{I_j}\chi_{I_j}(x) \quad\text{and}\quad
b(x):=f(x)-g(x)
\end{equation*}
($f=g+b$, $g$ and $b$ \emph{good} and  \emph{bad} (level $a$)-parts of $f$), verifying properties like the following:
\begin{equation*}
|g(x)|\le 2a \quad\text{(a.e.),} \qquad  \int_{I_j} b(x)\,dx=0\  \text{ and } \ |b|_{I_j}\le 4a \ \text{ for all $j$, etc.,} 
\end{equation*}
(see \cite[5.3.8]{grafakos}).

\smallskip
(2) We have seen that JLR \cite{rdf78} uses Theorem \ref{maxmartinteo} in his proof.
The a.e. convergence part (ii) of this theorem plays the role which is played by the differentiation theorem (DT) in a standard proof of the classic result of this type. But here this is not just a style option, we believe, because DT is not a priori assured. These issues are the main content of the next section.

\section{On differentiation of integrals in $\mathbb{T}^\omega$}
\label{sec:dif}
We will start establishing 
the concepts of differentiation basis and differentiation of integrals
adapted to the infinite torus space, which we will adopt in this section. 
\begin{defis} 
(\cite[Section 6.1]{bruckner}, \cite[Ch. 2]{guzmanyellow}.)
\label{defbasedepossel}
For every $y\in\mathbb{T}^\omega$ let $\mathcal{B}(y)$ be a collection of measurable sets of positive measure that contain (or whose topological closures contain) the point $y$. 

\smallskip
If $\{S_n\}_n\subset \mathcal{B}(y)$ and $\delta(S_n)\to 0$, we say that the sequence $S_n$ ``contracts to" $y$, and write  $S_n\Rightarrow y$.

Suppose that there exists at least a sequence $\{S_n\}\subset \mathcal{B}(y)$  such that $S_n\Rightarrow y$.

Let $\mathcal{B}:=\bigcup_{y\in\mathbb{T}^\omega}\mathcal{B}(y)$, and suppose that  $\mathcal{B}$ covers (a.e.) $\mathbb{T}^\omega$.
We call $(\mathcal{B}, \Rightarrow)$  a \emph{differentiation basis}\footnote{If every $B\in \mathcal{B}$ is an open set and if $x\in B\in\mathcal{B}$ then $B\in\mathcal{B}(x)$, $\mathcal{B}$ is called a \emph{Busemann-Feller basis}.} (DB).

\bigskip
\textsc{Examples} (The names are ours): 
\begin{gather*}
\mathcal{R}_0:=\left\{t+V_m\colon  m\in\mathbb{N}\text{,}\  t\in H_m \right\} \quad \text{(\emph{restricted Rubio de Francia basis}),} 
\\
\mathcal{R}:=\left\{y+V_m\colon  m\in\mathbb{N}\text{,}\  y\in \mathbb{T}^\omega \right\}\quad \text{(\emph{Rubio de Francia basis}),}\\
\mathcal{J}:=\left\{J\subset \mathbb{T}^\omega\colon \text{$J$ is an interval}\right\} \quad \text{(\emph{Jessen basis}), \cite{jessen1950,jessen1952}.}
\end{gather*}


Let $(\mathcal{B},\Rightarrow)$ be a differentiation basis in $\mathbb{T}^\omega$. Given $f\in L^1(\mathbb{T}^\omega)$, we define the \emph{upper and lower derivative} of $\int\!f$ with respect to $\mathcal{B}$ (and the Haar measure $m$) in the point $x\in\mathbb{T}^\omega$ by (without loss of generality we assume here that $f$ is a real function)
\begin{equation*}
\overline{D}\bigl( \textstyle\int\!f,x \bigr)=\displaystyle\sup_{\substack{\{B_n\}\subset\mathcal{B}\\B_n\Rightarrow x }}\bigl\{ \limsup_{n} 
f_{B_n}\bigr\}
\quad\text{and}\quad
\underline{D}\bigl( \textstyle\int\!f,x \bigr)=\displaystyle\inf_{\substack{\{B_n\}\subset\mathcal{B}\\B_n\Rightarrow x }}\bigl\{ \liminf_{n} f_{B_n}\bigr\},
\end{equation*}
respectively. When 
\begin{equation}\label{Bdiferintf}
\textstyle\overline{D}\Bigl(\int\!f,x \Bigr)=\underline{D}\Bigl( \int\!f,x \Bigr)
=f(x)\quad \text{a.e.}
\end{equation}
holds, we write $D\bigl(\int\!f,x \bigr)=f(x)$ and say that the basis $\mathcal{B}$ \emph{differentiates} $\int\!f$ and that the \emph{derivative} of $\int\!f$ is $f$. A necessary condition for \eqref{Bdiferintf} is that 
\begin{equation*}
\lim_{n\in\mathbb{N}} f_{B_n} =f(x)\quad \text{a.e.}
\end{equation*}
holds, for every sequence $\{B_n\}_{n\in\mathbb{N}}\subset  \mathcal{B}$ such that $B_n\Rightarrow x$. 

When \eqref{Bdiferintf} is satisfied for all $f\in L^\infty$ (resp. $f\in L^1(\mathbb{T}^\omega)$), we say that $\mathcal{B}$ \emph{differentiates} $L^\infty(\mathbb{T}^\omega)$ (resp. $L^1(\mathbb{T}^\omega)$). 
Note that $L^\infty(\mathbb{T}^\omega)\subset L^1(\mathbb{T}^\omega)$ and thus, if the basis $\mathcal{B}$ \emph{does not} differentiate $L^\infty(\mathbb{T}^\omega)$, then also does not differentiate  $L^1(\mathbb{T}^\omega)$. 
\end{defis}

Let $\mathcal{B}$ be a DB in $\mathbb{T}^\omega$. Suppose that the \emph{maximal operator} associated with $\mathcal{B}$ given by
$$
M^{\mathcal{B}}f(x)=\sup_{x\in B\in\mathcal{B}}|f|_B, \quad f\in L^1(m),
$$
is well defined
 (i.e., we suppose that for each $f\in L^1(m)$, $M^{\mathcal{B}}f$ is measurable).
The following result (due to de Guzm\'{a}n and Welland) holds, its proof is standard.
\begin {teor} \emph{(\cite[Thm. 1.1(a)]{guzmanywelland}.)} \label{gywell}
If the operator $M^{\mathcal{B}}$ is of weak type $(1,1)$, then the basis  $\mathcal{B}$ does differentiate $L^1(\mathbb{T}^\omega)$.
\end{teor}

\subsection{Differentiation Theorem on locally compact groups. The basis $\mathcal{R}$ in $\mathbb{T}^\omega$}
\label{subsec:DT}

\begin{teor} \emph{(\cite[Thm. 8]{rdf78}.)} \label{teor8jlr78ii} With the hypothesis of \emph{Theorem \ref{teor8jlr78}}, let $\mathcal{R}$ be the family formed by all sets of the form $yV_n$ with $y\in G$, $n=1,2,\ldots$.  If 
\begin{equation}\label{jlr78cond2}
\sup_n \,\frac{m(V_nV_n^{-1}V_n)}{m(V_n)}<\infty
\end{equation} 
holds, then $M^{\mathcal{R}}$ is weak type $(1,1)$ and strong $(p,p)$ for $1<p\le \infty$.
\end{teor}

As an immediate consequence
JLR gives 
the following result, which establishes a sufficient condition for the basis $\mathcal{R}$ to differentiate $L_{\text{loc}}^1(G)$ (in this setting, the notion of contraction of a sequence $(S_n)\subset \mathcal{R}$ to a point involves  $m(S_n)\to 0$). 

\begin{corol} (\cite[Corol. 5]{rdf78}.) 
\label{corol5jlr78}
Suppose, in addition to the hypothesis of  Theorem~\ref{teor8jlr78ii}, that  $V_n\subset U_n$ for a basis  $\{U_n\}_{n\in\mathbb{N}}$ of neighbourhoods of $e$. Then:
$$
\lim_{\substack{x\in R\in\mathcal{R}\\ m(R)\to 0}} f_R =f(x) \quad \text{(a.e.)}
$$
for any locally integrable function $f$.
\end{corol}

In the case in which $G$ is the compact group $\mathbb{T}^\omega$ 
and  $\mathcal{R}$ is our Rubio de Francia basis, the  condition \eqref{jlr78cond2} does not hold, because e.g. for $n\ge 1$,
\begin{gather*}
V_{n^2}=\Bigl(0,\frac1{2^n} \Bigr)^n\times \mathbb{T}^{n,\omega}\\\intertext{and}
V_{n^2}-V_{n^2}+V_{n^2}=\left(\Bigl[0,\frac1{2^{n-1}} \Bigr)\cup \Bigl(\frac{2^n-1}{2^n},1\Bigr)\right)^n\times \mathbb{T}^{n,\omega},
\end{gather*} 
so that
$$
\frac{m(V_{n^2}-V_{n^2}+V_{n^2})}{m(V_{n^2})}=\frac{(3/2^n)^n}{(1/2^{n})^n}=3^n,\quad\text{and}\quad
\sup_n\frac{m(V_{n}-V_{n}+V_{n})}{m(V_{n})}=\infty.
$$ 
Therefore, we can not guarantee (in principle) the result of Theorem  \ref{teor8jlr78ii} for the Rubio de Francia basis $\mathcal{R}$. The additional sufficient condition of the Corollary \ref{corol5jlr78} is satisfied because, for instance, the family $\{V_n-V_n\}_{n\in\mathbb{N}}$ is a basis of (symmetric) neighbourhoods of 0 in $\mathbb{T}^\omega$.
On the Rubio de Francia basis $\mathcal{R}$ the following questions remain open:

\begin{itemize} 
\item Does the converse of Theorem \ref{gywell} for the basis  $\mathcal{R}$ in $\mathbb{T}^\omega$ hold? (in de Guzm\'{an} and Welland theorem in $\mathbb{R}^n$ \cite[Thm. 1.1(b)]{guzmanywelland} the BD $\mathcal{B}$ is required to be homothecy invariant).
\item Is the operator $M^{\mathcal{R}}$ weak type (1,1)?\footnote{We owe this question to Sheldy J. Ombrosi.} 
\item Does $\mathcal{R}$ differentiate $L^\infty(\mathbb{T}^\omega)$?
\end{itemize}
Customizing Jessen's proof for the basis $\mathcal{J}$, we will prove below (Subsection \ref{subsec:neg})  that a certain basis $\mathcal{R}^\ast$  slightly wider than $\mathcal{R}$ does not differentiate $L^\infty(\mathbb{T}^\omega)$.

\subsection{Bases $\mathcal{R}_0$ and $\mathcal{J}$}
\begin{defi} (\cite[p. 153]{saks}, \cite[VII.43]{munroe}.) \label{defsnetfina}
Let $\{\mathcal{M}_n\}_{n\in\mathbb{N}}$ be a monotonic sequence of nets in $\mathbb{T}^\omega$, and $\mathcal{M}:=\bigcup_n \mathcal{M}_n$.
For  each $y\in \mathbb{T}^\omega$ and each $k$, write $I_y^{(k)}$ for the unique element of the net $\mathcal{M}_k$ which contains $y$.
We say that the sequence is \emph{indefinitely fine} if for each $x\in \mathbb{T}^\omega$ and each $\varepsilon>0$ there is $n_0\in\mathbb{N}$ such that $\delta\bigl(I_x^{(n_0)}\bigr)<\varepsilon$. 
\end{defi}

Then, $I_x^{(n)} \Rightarrow x$, and $\mathcal{M}$ is a differentiation basis. 
The following result holds.
\begin{teor} \emph{(\cite[43.7]{munroe}, cf. also \cite[\S 9]{jessen}, \cite[15.7]{saks}.)} \label{diferdeL1} If $\{\mathcal{M}_n\}_{n\in\mathbb{N}}$ is a monotonic sequence of nets  indefinitely fine, then the basis $\mathcal{M}$ differentiates $L^1(\mathbb{T}^\omega)$.
\end{teor}

At the end of \cite{rdf78}, JLR pointed out, in the setting of the locally compact group $G$, that if $\mathcal{R}$  
is defined to be only consisting of the sets  
$tV_n$ ($t\in H_n$), $n=1,2,\ldots$, then Theorem~\ref{teor8jlr78ii} %
is valid without assumptions \eqref{jlr78cond1} and \eqref{jlr78cond2}. %
In order to corroborate this statement with an example,  we provide an immediate consequence of Theorem~\ref{diferdeL1}. It is, on the other hand, an immediate consequence of Theorem \ref{gywell}, because the maximal operator $M^{\mathcal{R}_0}$ is of weak type $(1,1)$.

\begin{corol} 
\label{cor:RDFr}
The basis $\mathcal{R}_0= \left\{t+V_m\colon  m\in\mathbb{N}, t\in H_m\right\}$ 
does differentiate $L^1(\mathbb{T}^\omega)$.
\end{corol}
\begin{proof} The basis $\mathcal{R}_0$ is the union, for $m\in\mathbb{N}$,  
of the monotonic sequence of nets $\{\mathcal{N}_m\}$,  where 
$\mathcal{N}_m= \left\{t+V_m\colon t\in H_m\right\}$. This sequence is indefinitely fine, because
if $I\in \mathcal{N}_m$ and $(n-1)^2<m\le n^2$ ($n\ge 2$), it is easily seen  that 
$$
\delta(I)\le \sum_{j=1}^{n-1} \frac1{2^{n-1+j}}+\frac1{2^{n+1}}
+\sum_{j=n+1}^\infty\frac1{2^j}<\frac7{2^{n+1}}. 
$$
\end{proof}

Any subfamily of a basis that differentiates $L^1$ and
which is in turn a basis of differentiation, also differentiates $L^1$.
In particular, the subfamily of cubic intervals of the
base $\mathcal{R}_0$ that Saks already considered \cite[p. 158]{saks} 
(see \cite[p. 28]{bruckner})
$$
\mathcal{S}:=\bigcup_{m=1}^\infty \mathcal{S}_m\quad \text{ where }
\quad \mathcal{S}_m:=\left\{t+V_{m^2}\colon t\in H_{m^2}\right\},
$$
also differentiates $L^1(\mathbb{T}^\omega)$.

The question (posed by A. Zygmund, see \cite[p. 55]{jessen1950}) about the differentiation of integrals in $L^1(\mathbb{T}^\omega)$ with respect to this basis $\mathcal{J}$ of all the intervals of $\mathbb{T}^\omega$ was answered negatively around 1950  by Jessen \cite{jessen1950,jessen1952}.
The counterexample proposed by Jessen refers to the characteristic function of certain measurable set of positive measure, so in fact he proves that the 
basis $\mathcal{J}$ \emph{does not} differentiate even $L^\infty(\mathbb{T}^\omega)$.
It is indeed a curious phenomenon, since the basis formed by the \emph{intervals} (i.e., the parallelepipeds of edges parallel to the coordinate axes) of $\mathbb{T}^n$ does differentiate $L^\infty(\mathbb{T}^n)$ for all $n\in\mathbb{N}$ \cite[p. 74]{guzmanyellow}.

\subsection{A negative result of differentiation on $\boldsymbol{\mathbb{T}^\omega}$: The  basis $\boldsymbol{\mathcal{R}^\ast}$}
\label{subsec:neg}

\quad 

We begin with questions of nomenclature and notation to briefly represent some dyadic sets in $\mathbb{T}^\omega$.


\begin{defis} 
For $m$, $q\in\mathbb{N}$, $m\ge 2$, and $q\le m$, write $\widetilde{\square}_{m,q}:=\bigl(0,\tfrac1{2^m} \bigr)^q$ and
$$\square_{m,q}:=\widetilde{\square}_{m,q}\times\mathbb{T}^{q,\omega}.$$
Call $\square_{m,q}$ the \emph{$(m,q)$-cube}. 
E.g.,
$V_{m^2}=\square_{m,m}=\widetilde{\square}_{m,q}\times \bigl(0,\frac1{2^m}\bigr)^{m-q}\times\mathbb{T}^{m,\omega}$, for all $m\ge q$.

Consider  in $\mathbb{T}^{\omega}$, for $j\in\mathbb{N}$, the translation $\tau_j^{m}$ which adds $\tfrac1{2^m}$ to the coordinate $x_j$.
Define the sets (we call them \emph{sacks} of the corresponding cubes)  
$$S(\square_{m,q}):=\square_{m,q}\cup\bigl(\cup_{j=1}^q\tau_j^{m}(\square_{m,q})\bigr)$$
and, for $y\in\mathbb{T}^\omega$, 
$S(y+\square_{m,q})=y+S(\square_{m,q})$.

\smallskip
On the other hand, 
write $W_{m,r}:=\square_{m,m}\cup \tau_r^{m}(\square_{m,m})$ $(1\le r\le m)$. Call $W_{m,r}$ a \emph{double $(m,m)$-cube}.
We have $W_{m,r}=\widetilde{W}_{m,r}\times\mathbb{T}^{m,\omega}$, where
$$
\widetilde{W}_{m,r}:=\prod_{j=1}^m (1+\delta_{rj})\cdot \bigl(0,\tfrac1{2^m} \bigr)
\qquad\text{(Kronecker's delta).}
$$
E.g., $W_{m,m}=V_{m^2-1}$, but  $W_{m,j}\notin\mathcal{R}$ when $1\le j\le m-1$ (see Figure \ref{CZ_fig2}). 

\smallskip
We define the \emph{extended Rubio de Francia basis} to be the collection
\begin{equation}\label{extendedrdfbasis}
\mathcal{R}^\ast:=\mathcal{R}\cup\{y+W_{m,r}\colon y\in\mathbb{T}^\omega,\ m\in\mathbb{N},\ m\ge 2,\ 1\le r\le m\}.
\end{equation}
\end{defis}
\begin{figure}[ht]
\begin{center}
\quad\hspace*{-10pt}\includegraphics[scale=.6]{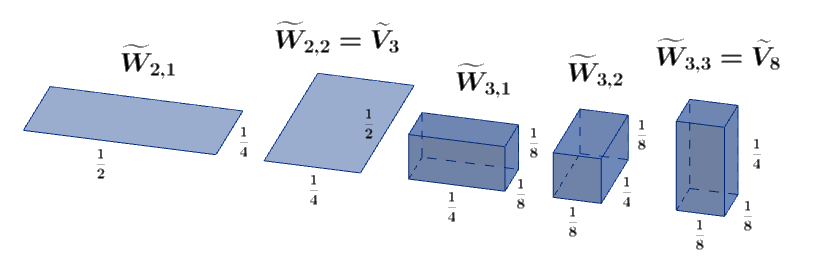}
\caption{First members of the family $\widetilde{W}_{m,r}$. 
}\label{CZ_fig2}
\end{center}
\end{figure}

\begin{lema} \label{lemma18} Let $Q\in\{y+\square_{m,q}\colon y\in\mathbb{T}^\omega\}$. For each point $x\in S(Q)$ there is an interval $I_x\in\mathcal{R}^\ast$ such that
$$
\frac{m( I_x\cap Q)}{m(I_x)}\ge \frac 12.
$$
\end{lema}
\begin{figure}[ht]
\begin{center}
\includegraphics[scale=.4]{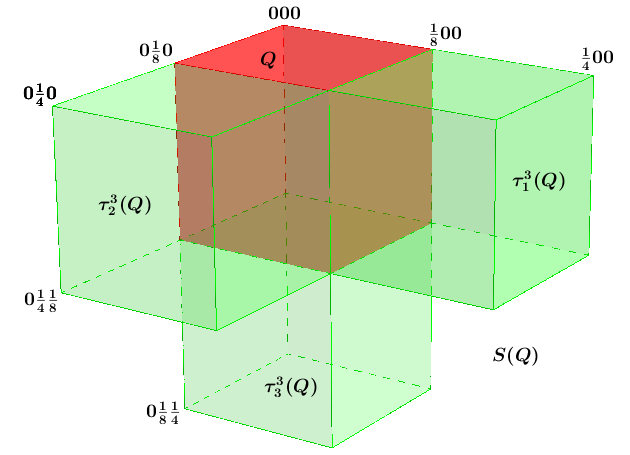}
\caption{The sack of the cube $Q=\square_{3,3}$.}\label{CZ_figsack}
\end{center}
\end{figure}
\begin{proof} First consider any case in which $y=\overline{0}$, $Q=\square_{m,q}$. Then, if $x\in Q$ we can take the interval
$I_x^{0}=\widetilde{\square}_{m,q}\times\bigl(0,\frac1{2^m}\bigr)^{m-q}\times\mathbb{T}^{m,\omega}=V_{m^2}\in\mathcal{R}$. We have $I_x^{0}\subset Q$,
and
$$
\frac{m(I_x^{0}\cap Q)}{m(I_x^{0})}=1.
$$
Otherwise, if $x\in\tau_r(Q)$ for any $r$, $1\le r\le q$, 
define ($\delta_{rj}$ is Kronecker's delta)
$$
D_{m,r}^{q}:=\prod_{j=1}^q (1+\delta_{rj})\cdot \bigl(0,\tfrac1{2^m} \bigr)\subset \mathbb{T}^q.
$$
Then, let us take
$I_x^{0}=D_{m,r}^{q}\times\bigl(0,\tfrac1{2^m}\bigr)^{m-q}\times\mathbb{T}^{m,\omega}$.
We have $I_x^{(0)}=W_{m,r}\in\mathcal{R}^\ast$, $I_x^{0}\cap Q=\widetilde{\square}_{m,q}\times\bigl(0,\frac1{2^m}\bigr)^{m-q}\times\mathbb{T}^{m,\omega}=\square_{m,m}$, and
$$
\frac{m(I_x^{0}\cap Q)}{m(I_x^{0})}=\frac{m(\square_{m,m})}{m(W_{m,r})}=\frac12.
$$
In a general case in which $y\ne \overline{0}$, take $I_x=y+I_x^{0}$.  The lemma follows.
\end{proof}

\begin{lema} \label{lemajess52} Let $n\ge 2$ be a fixed integer. 
It is possible to find in $\mathbb{T}^\omega$ an enumerable family of pairwise disjoint intervals 
$\{Q_\alpha\}_{\alpha\in A_n}$, $Q_\alpha=y(\alpha)+\square_{m(\alpha),n}$ ($y(\alpha)\in \mathbb{T}^\omega$,  $m(\alpha)\ge n$), 
with the sets $S(Q_\alpha)$ also pairwise disjoint, and:
\begin{enumerate}
\item If $C_n:=\bigcup_{\alpha} Q_\alpha$, and $N_n:=\mathbb{T}^\omega\setminus \bigl(\cup_{\alpha} S(Q_\alpha)\bigr)$, then $m(C_n)=\frac 1{n+1}$, and  $m(N_n)=0$.
\item For every $x\notin N_n$ there exists an interval  $I_x^{n}\in \mathcal{R}^\ast$ 
(whose first $n$ edges are $\le 1/2^{n-1}$; in fact, $I_x^{n}$ is a translate either of a cube $V_{n^2}$, or of a double cube $W_{n^2-1,r}$) 
such that
$$\frac{m(I_x^{n}\cap C_n)}{m(I_x^{n})}\ge \frac 12.$$
\end{enumerate}

\end{lema}

\begin{figure}[ht]
\begin{center}
\includegraphics[scale=.6]{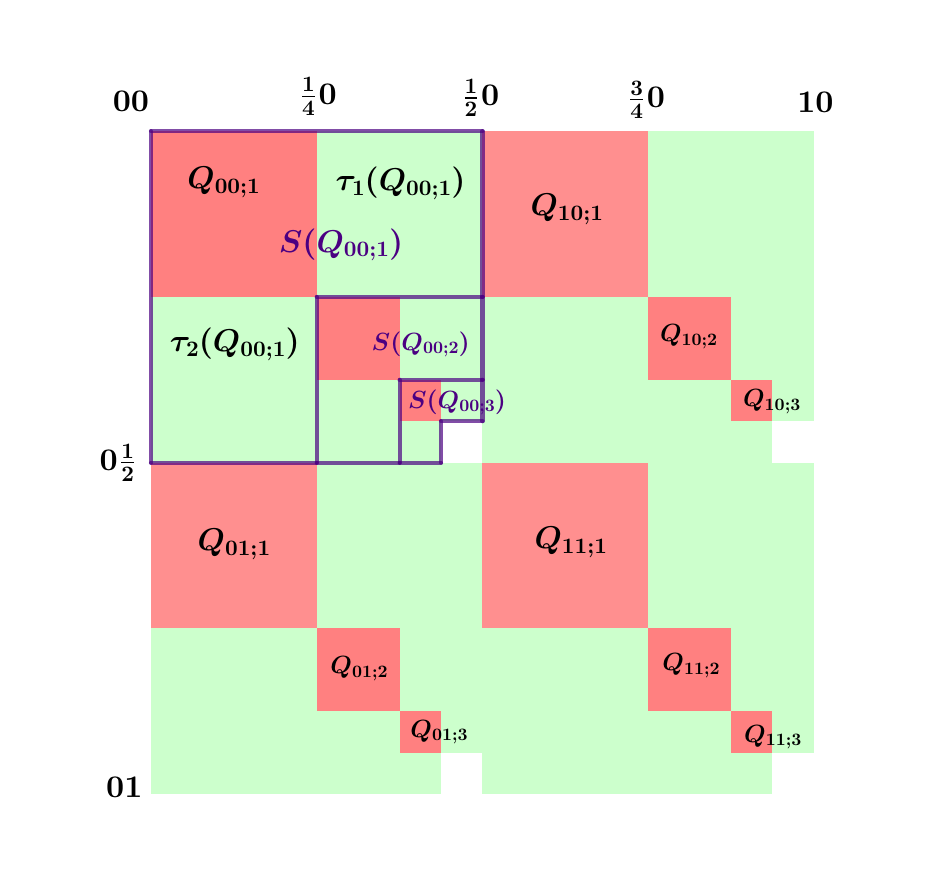}
\includegraphics[scale=.35]{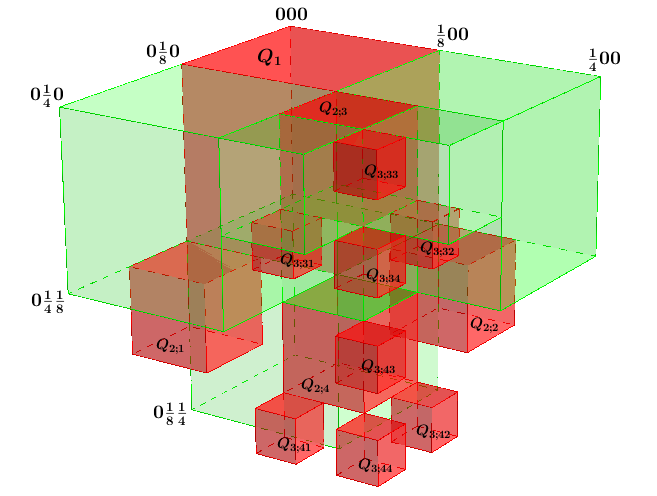}
\caption{The construction in Lemma \ref{lemajess52}, $n=2$ and 3.} 
\label{CZ_fig3}
\end{center}
\end{figure}

\begin{proof} (In what follows, we write $\boldsymbol{t_1\ldots,t_n\overline{0}}$ for  the point
$\bigl(t_1,\ldots,t_n,0^{(n}\bigr)\in\mathbb{T}^\omega$.) 
In general, for each $n\ge 2$, we consider the division of $\mathbb{T}^\omega$ into $2^{n(n-1)}$ open cubes (call them \emph{$0$-cells}) of edge $1/2^{n-1}$ by the \lq\lq{}hyperplanes\rq\rq{} 
$$
x_i=\frac j{2^{n-1}} \qquad(i=1,\ldots,n;\ j=0,1,\ldots, 2^{n-1}-1).
$$
The $0$-cells have the form
$$
I_{i_1\ldots i_n}^{(n)}=\boldsymbol{\tfrac {i_1}{2^{n-1}}\cdots \tfrac {i_n}{2^{n-1}}\overline{0}}+\bigl(0,\tfrac1{2^{n-1}}\bigr)^n\times\mathbb{T}^{n,\omega}, 
$$
where $(i_1,\ldots, i_n)\in\{0,1,\ldots,2^{n-1}-1\}^n$.

Each one of the $0$-cells is firstly subdivided in $2^n$ open cubic intervals of edge length $1/2^n$ ($1$-\emph{cells}).  
As an example, the $1$-cells of the $I_{0\ldots 0}^{(n)}$ $0$-cell have the form 
$$
\boldsymbol{\tfrac {j_1}{2^{n}}\cdots \tfrac {j_n}{2^{n}}\overline{0}}+\square_{n,n},
$$
where $(j_1,\ldots, j_n)\in\{0,1\}^n$.

Among these $1$-cells, we call $\boldsymbol{0\ldots0\overline{0}}+\square_{n,n}$ the \emph{principal $1$-cell} and denote it by $Q_{0\ldots 0;1}^{(n)}$ (abbreviated by $Q_1^{(n)}$).  
The sack of $Q_1^{(n)}$ is
$$
S(Q_1^{(n)})=Q_1^{(n)}\cup\Bigl(\bigcup_{j=1}^n \bigl(\boldsymbol{\tfrac{\delta_{1j}}{2^n}\cdots \tfrac{\delta_{nj}}{2^n}\overline{0}} +Q_1^{(n)}\bigr) \Bigr) \quad\text{(Kronecker's $\delta_{ij}$),}
$$
and we have $m(S(Q_1^{(n)}))=(n+1)\cdot m(Q_1^{(n)})$.

In the $0$-cell $I_{0\ldots 0}^{(n)}$, apart from the $(n+1)$ $1$-cells which form the sack $S(Q_1^{(n)})$, there remain other $2^n-(n+1)$ $1$-cells. In each one of these is carried out a subdivision into $2^n$ open cubes of edge $1/2^{n+1}$ ($2$-\emph{cells}), one of which is the principal $2$-cell $Q_{2;\beta}^{(n)}=\boldsymbol{y_\beta}+\square_{n+1,n}$, with an appropriate $\boldsymbol{y_\beta}\in\mathbb{T}^\omega$. There are $2^n-(n+1)$ principal $2$-cells for each principal $1$-cell. Now, we forget the sacks of the principal $2$-cells, and in all the remaining $2$-cells we proceed inductively.

The family  $\{Q_\alpha^{(n)}\}_{\alpha\in A_n}$ of the statement of the Lemma is formed by all the principal $k$-cells ($k\ge 1$) in this construction. The indicial set $A_n$ can be defined explicitly (look at Figure \ref{CZ_fig3}), but this is not essential, and for clarity of exposition we will avoid doing it.
On the other hand, from the inductive definition it is immediate that 
$$
Q_\alpha^{(n)}\cap Q_\beta^{(n)}=\emptyset \quad \text{ and }\quad
S(Q_\alpha^{(n)})\cap S(Q_\beta^{(n)})=\emptyset \quad \text{ for all $\alpha\ne \beta$ ($n\in\mathbb{N}$, $\alpha,\beta\in A_n$). }
$$

\smallskip
(1) Let now $C_n:=\bigcup_{\alpha\in A_n} Q_\alpha^{(n)}$.
Then, 
$$
m(C_n)=\sum_{\alpha\in A_n}m(Q_\alpha^{(n)})=\frac{2^{n(n-1)}}{2^{n^2}}\sum_{k=0}^\infty\Bigl( \frac{2^n-(n+1)}{2^n} \Bigr)^k =\frac1{n+1},
$$
and
$$
m\Bigl(\bigcup_{\alpha\in A_n} S(Q_\alpha^{(n)})\Bigr)=\sum_{\alpha\in A_n}m(S(Q_\alpha^{(n)}))
=(n+1)\sum_{\alpha\in A_n}m(Q_\alpha^{(n)})
=(n+1) m(C_n)=1,
$$
from which, denoting $N_n=\mathbb{T}^\omega\setminus \Bigl(\bigcup_{\alpha\in A_n} S(Q_\alpha^{(n)})\Bigr)$, we have $m(N_n)=0$.

\smallskip
(2) Let $x\notin N_n$. Then, there exists $\alpha_0\in A_n$ such that $x\in S(Q_{\alpha_0}^{(n)})$. Applying Lemma \ref{lemma18}, there exists an interval $I_x\in\mathcal{R}^\ast$ (in fact, either a cube of edge not greater than $1/2^n$, or a double cube) such that
$$
\frac{m(I_x\cap Q_{\alpha_0}^{(n)} )}{m(I_x)}=\frac{m(I_x\cap C_n)}{m(I_x)}\ge \frac 12, 
$$
as required.
\end{proof}

\begin{teor} \label{contradif}
The Rubio de Francia extended basis $\mathcal{R}^\ast$  does not  differentiate $L^\infty(\mathbb{T}^\omega)$. 
\end{teor}
\begin{proof} (This argumentation is taken from Jessen \cite{jessen1952}.) 

Choose an increasing sequence of positive integers $(n_p)_{p=1}^\infty$ ($n_1\ge 2$) such that $\sum_p 1/(n_p+1)\le 3/4$. Then, the union $C:=\bigcup_{p=1}^\infty C_{n_p}$ is a measurable set whose measure satisfies $0<\frac1{n_1+1}\le m(C)\le 3/4$, and the union  $N:=\bigcup_{p=1}^\infty N_{n_p}$ is a null measurable set, since for each $p$, the set $N_{n_p}$ is measurable  and  $m(N_{n_p})=0$.

If $x\notin N$, then $x\in\bigcup_{\alpha}S(Q_\alpha^{(n_p)})$ for every $p$, and thus there exists a sequence of indexes $(\alpha_p)_{p=1}^\infty$ such that $x\in S(Q_{\alpha_p}^{(n_p)})$ for each $p$. Then, applying Lemma \ref{lemajess52}(2), for each $p$ there exists an interval 
$I_x^{(p)}\in\mathcal{R}^\ast$ such that $\delta(I_x^{(p)})\le \delta(W_{n_p,1})<3/2^{n_p}$ (consequently these intervals $I_x^{(p)}$ form a sequence of $\mathcal{R}^\ast$ contracting to the point $x$), and
$$
\frac{m(C\cap I_x^{(p)})}{m(I_x^{(p)})}\ge \frac 12.
$$
Consider the characteristic function $\chi_C$. For all $x\in\mathbb{T}^\omega\setminus N$ (i.e., a.e. in $\mathbb{T}^\omega$), we have
$$
\limsup_{p\to\infty} \frac1{m(I_x^{(p)})}\int_{I_x^{(p)}}\chi_C(y)\,dy=\limsup_{p\to\infty} \frac{m(C\cap I_x^{(p)})}{m(I_x^{(p)})}\ge \frac 12,
$$
 which immediately implies that $\overline{D}\bigl( \int\!\chi_C,x \bigr)\ge\frac12$ for almost all $x\in\mathbb{T}^\omega$. 
But $\chi_C(x)=0<\frac12$ for all $x\notin C$, a set of measure $\ge 1/4$.
It follows that $\mathcal{R}^\ast$ does not differentiate $\int\!\chi_C$.
\end{proof}

\medskip
\noindent\textbf{Remarks.} 

\smallskip
(1) The proof of Theorem \ref{contradif} in fact shows that the subfamily extracted from $\mathcal{R}^\ast$ which is formed by the cubes $\{y+V_{m^2}\colon y\in\mathbb{T}^\omega\text{,}\ m\ge 2\}$ and the double cubes $\{y+W_{m,r}\colon y\in\mathbb{T}^\omega\text{,}\ m\ge 2\text{,}\ 1\le r\le m\}$ (this subfamily is not contained in the Rubio de Francia basis $\mathcal{R}$) does not  differentiate $L^\infty(\mathbb{T}^\omega)$.

\smallskip
(2) The question whether the DB formed only by the cubes $\{y+V_{m^2}\colon y\in\mathbb{T}^\omega\text{,}\ m\ge 2\}$  does  differentiate $L^\infty(\mathbb{T}^\omega)$ (with our notion of contraction of a sequence to a point) remains open for us at the moment. 

\smallskip
(3)  Dieudonn\'{e} \cite{dieudonne} also proves that the basis of intervals in $\mathbb{T}^\omega$ (in fact, the subfamily of cubic intervals in $[0,1]^\omega$), does not differentiate $L^\infty(\mathbb{T}^\omega)$ (see \cite[p. 28]{bruckner}). But Dieudonn\'{e} works with the notion of contraction to a point for generalized sequences in the Moore-Smith sense $\{S_\alpha\}_{\alpha\in D}\subset \mathcal{B}(y)$, being $D$ a directed set \cite[p. 81-86]{kelley}, as we explain next:

Let $\mathcal{F}$ be the set of finite subsets of $\mathbb{N}$. For each  $J\in\mathcal{F}$ we consider
$$
\mathbb{T}^\omega=\mathbb{T}^J\times \mathbb{T}^{J,\omega}
$$
in such a way that, if $x\in \mathbb{T}^\omega$, $x=(x_J,x_{J'})$ with $x_J\in\mathbb{T}^J$ and $x_{J'}\in\mathbb{T}^{J,\omega}$.
Dieudonn\'{e} deals with the DB $\mathcal{D}=\bigcup_{x\in\mathbb{T}^\omega}\mathcal{D}(x)$ where $\mathcal{D}(x)$ is the net (according to the set $\mathbb{N}\times\mathcal{F}$ directed by the order relation
 $(n_1,J_1)\le(n_2,J_2)$ if and only if $n_1\le n_2$ and $J_1\subseteq J_2$) that consists of the cubic intervals
\begin{equation}\label{cubosdedieu}
V_{n,J}(x)=\widetilde{V}_{n,J}(x_J)\times \mathbb{T}^{J,\omega}, \quad (n\in\mathbb{N}; \ J\in\mathcal{F}),
\end{equation}
where $\widetilde{V}_{n,J}(x_J)\subset \mathbb{T}^J$ is the cube of center $x_J$ and side $1/n$.

Dieudonn\'{e} defines a measurable set for whose characteristic function $f$, the means
$f_{V_{n,J}(x)}$ cannot converge a.e. to $f(x)$ according to the directed set  $\mathbb{N}\times\mathcal{F}$.

\smallskip
(4) A differentiation basis $\mathcal{B}$ satisfies the \emph{density property} if $\mathcal{B}$ differentiates $\chi_E$ for each measurable set $E$, i.e.  for almost every $x\in\mathbb{T}^\omega$ we have, if $\{I_k\}$ is any arbitrary sequence of $\mathcal{B}(x)$ contracting to $x$,
$$
\lim_{k\to\infty}\frac{m(E\cap I_k)}{m(I_k)}=\chi_E(x)
$$
(\cite[p. 227]{busfeller}, \cite[p. 30]{hayes_pauc}, \cite[III.1]{guzmanyellow}).
From our proof of Theorem \ref{contradif} it follows that the basis $\mathcal{R}^\ast$  does not satisfies the density property.

In fact, it holds  the following result (which for the space  $\mathbb{R}^n$ can be found, for instance, in  \cite[III, Thm. 1.4]{guzmanyellow}):
\emph{The basis  $\mathcal{B}$ differentiates $L^\infty(\mathbb{T}^\omega)$ if and only if satisfies the density property} \cite[Num. 11, C$\Leftrightarrow$D]{depossel}.



\end{document}